\documentclass[11pt]{amsart}
\usepackage[centertags]{amsmath}
\usepackage{amsfonts}
\usepackage{amssymb}
\usepackage{amsthm}
\usepackage{newlfont}
\usepackage{amscd}
\usepackage{mathrsfs}
\usepackage[all,cmtip]{xy}
\usepackage{tikz-cd}
\tikzcdset{arrow style=Latin Modern, row sep/normal=1.35em, column sep/normal=1.8em, cramped}
\usepackage{tikz}
\usepackage[pagebackref=false,colorlinks]{hyperref}
\usepackage{array}

\definecolor{mycolor}{HTML}{F7F8E0}
\definecolor{myorange}{RGB}{245,156,74}
\definecolor{cadetgrey}{rgb}{0.57, 0.64, 0.69}
\definecolor{calpolypomonagreen}{rgb}{0.12, 0.3, 0.17}

\hypersetup{pdffitwindow=true,linkcolor=calpolypomonagreen,citecolor=calpolypomonagreen,urlcolor=calpolypomonagreen}
\usepackage{cite}
\usepackage{fancyhdr}
\usepackage{stmaryrd}

\usepackage[OT2,OT1]{fontenc}
\newcommand\cyr{%
\renewcommand\rmdefault{wncyr}%
\renewcommand\sfdefault{wncyss}%
\renewcommand\encodingdefault{OT2}%
\normalfont
\selectfont}
\DeclareTextFontCommand{\textcyr}{\cyr}

\newcommand{\Mod}[1]{\ (\text{mod}\ #1)}

\usepackage[toc, page] {appendix}

\numberwithin{equation}{section}

\newtheorem{thm}{Theorem}[section]

\newtheorem{cor}[thm]{Corollary}

\newtheorem{prop}[thm]{Proposition}

\newtheorem{assu}[thm]{Assumption}

\theoremstyle{definition}

\newtheorem{rem}[thm]{Remark}

\begin{document}

\title{On the anticyclotomic Mazur--Tate conjecture for elliptic curves with supersingular reduction}
\author{Chan-Ho Kim}
\address{
Department of Mathematics and Institute of Pure and Applied Mathematics,
Jeonbuk National University,
567 Baekje-daero, Deokjin-gu, Jeonju, Jeollabuk-do 54896, Republic of Korea
}
\email{chanho.math@gmail.com}
\thanks{This research was partially supported 
by a KIAS Individual Grant (SP054103) via the Center for Mathematical Challenges at Korea Institute for Advanced Study,
by the National Research Foundation of Korea(NRF) grant funded by the Korea government(MSIT) (No. 2018R1C1B6007009, 2019R1A6A1A11051177), and
by Global-Learning \& Academic research institution for Master’s$\cdot$Ph.D. Students, and Postdocs (LAMP) Program of the National Research Foundation of Korea (NRF) funded by the Ministry of Education (No. RS-2024-00443714).
}
\date{\today}
\begin{abstract}
In this short note, we study the anticyclotomic analogue of the ``weak" main conjecture of Mazur--Tate on Fitting ideals of Selmer groups for elliptic curves with supersingular reduction.
\end{abstract}
\maketitle

\setcounter{tocdepth}{1}
\section{Introduction}
\subsection{The statement of the main result}
Let $E$ be an elliptic curve of conductor $N$ over $\mathbb{Q}$ and $p \geq 5$ be a prime of supersingular reduction for $E$, so $a_p(E) = 0$.

Let $K$ be an imaginary quadratic field with $(D_K, Np) = 1$ where $p$ splits and and write $p = \mathfrak{p} \cdot \overline{\mathfrak{p}}$ in $K$.  
Write 
$$N = N^+ \cdot N^-$$
 where a prime divisor of $N^+$ splits in $K$ and a prime divisor of $N^- $ is inert in $K$.
\begin{assu}[odd] \label{assu:definite}
Assume that $N^-$ is the square-free product of an odd number of primes.
\end{assu}
Let $K_\infty$ be the anticyclotomic $\mathbb{Z}_p$-extension of $K$ and $K_n$ be the subextension of $K$ in $K_\infty$ of degree $p^n$ for $n \geq 0$.
Let $\Lambda_n = \mathbb{Z}_p[\mathrm{Gal}(K_n/K)] \simeq \mathbb{Z}_p[X]/( (1+X)^{p^n} - 1 )$ be the finite layer Iwasawa algebra.
Under Assumption \ref{assu:definite}, denote by
$$\theta_n(E/K) = \sum_{\sigma \in \mathrm{Gal}(K_n/K)} a_{\sigma} \cdot \sigma \in \Lambda_n$$
 Bertolini--Darmon's theta element of $E$ over $K_n$ reviewed in Section \ref{sec:bertolini-darmon-elements}.
Let $\iota$ be the involution on $\Lambda_n$ defined by inverting group-like elements.
Then we have $\iota( \theta_n(E/K) ) = \sum_{\sigma \in \mathrm{Gal}(K_n/K)} a_{\sigma} \cdot \sigma^{-1}$
and define
$$L_p(E/K_n) = \theta_n(E/K) \cdot \iota( \theta_n(E/K) ).$$
\begin{assu} \label{assu:ram-CR}
We assume the following hypotheses throughout this article:
\begin{itemize}
\item[(Im)] The mod $p$ representation $\overline{\rho} : G_{\mathbb{Q}} = \mathrm{Gal}(\overline{\mathbb{Q}}/\mathbb{Q}) \to \mathrm{GL}_2(\mathbb{F}_p)$ is irreducible. It is surjective if $p = 5$.
\item[(ram)] $K_\infty / K$ is totally ramified at every prime lying above $p$.
\item[(CR)] If a prime $\ell$ divides $N^-$ and $\ell^2 \equiv 1 \pmod{p}$, then $\overline{\rho}$ is ramified at $\ell$.
\end{itemize}
\end{assu}
The goal of this article is to prove the following anticyclotomic analogue of the ``weak" main conjecture of Mazur and Tate \cite[Conj. 3]{mazur-tate}.
\begin{thm} \label{thm:main}
Under (odd), (Im), (ram), and (CR), we have
$$L_p(E/K_n)  \in \mathrm{Fitt}_{\Lambda_n} \left( \mathrm{Sel}(K_n, E[p^\infty])^\vee \right) .$$
where $(-)^\vee$ means the Pontryagin dual.
\end{thm}
Let $\chi : \mathrm{Gal}(K_n/K) \to \overline{\mathbb{Q}}^\times_p$ be a character and we extend it to $\Lambda_n$ linearly.
Let $I_\chi := \mathrm{ker} \left( \chi :\Lambda_n \to \overline{\mathbb{Q}}_p \right) \subseteq \Lambda_n$ be the augmentation ideal at $\chi$
and $( E(K_n) \otimes \overline{\mathbb{Q}}_p)^{\chi}$ be the $\chi$-isotypic subspace of $E(K_n) \otimes \overline{\mathbb{Q}}_p$.
Thanks to \cite[Prop. 3]{mazur-tate}, we obtain the ``weak" vanishing conjecture.
\begin{cor} \label{cor:weak-vanishing}
Under (odd), (Im), (ram), and (CR), we have
$$L_p(E/K_n)  \in I^{r_{\chi}}_\chi/I^{r_{\chi} +1}_\chi$$
where $r_{\chi} = \mathrm{dim}_{\overline{\mathbb{Q}}_p}( E(K_n) \otimes \overline{\mathbb{Q}}_p)^{\chi}$.
\end{cor}
In particular, $L_p(E/K_n)$ encodes an upper bound of $\mathrm{rk}_{\mathbb{Z}}E(K)$ by considering the trivial character.
In the ordinary case, Corollary \ref{cor:weak-vanishing} can be deduced from the one-sided divisibility of the anticyclotomic main conjecture more directly.
In \cite{darmon-refined-bsd}, Darmon studied a similar ``weak" vanishing conjecture over ring class extensions of conductor prime to $p$ 
when (odd) breaks down ($N^- = 1$).
\begin{rem}
In \cite{mazur-tate}, Mazur and Tate formulated ``Iwasawa theory for elliptic curves over finite abelian extensions" and several related conjectures. 
This finite layer setting is much more delicate since the linear algebra techniques for the usual Iwasawa theory (over $\mathbb{Z}_p$-extensions) are not applicable and we cannot ignore ``finite errors" in the computation.
The formulation of the ``main conjecture" also becomes more subtle since the notion of characteristic ideals is not available.
The ``weak" main conjecture is formulated in terms of Fitting ideals. See \cite[Conj. 0.3]{kurihara-invent} for the strong version. Here, the strong version means an equality.
\end{rem}
\subsection{The strategy of proof}
The strategy of proof is similar to that in \cite{kim-kurihara}. We begin with the one-sided divisibility of the signed main conjecture.
Since signed Selmer groups have no proper $\Lambda$-submodules of finite index, we can replace the characteristic ideals by the Fitting ideals.
Then we can project the divisibility to finite layers since Fitting ideals are compatible with base change.
On the algebraic side, we compare the Fitting ideals of dual signed Selmer groups with the Fitting ideals of dual Selmer groups at finite layers. The difference essentially comes from the failure of the control theorem for Selmer groups in the supersingular case.
On the analytic side, we compare the signed $p$-adic $L$-functions at finite layers with the square of Bertolini--Darmon's theta elements, and we observe that the same difference occurs. The conclusion follows from the comparison of these differences.

One may ask whether Theorem \ref{thm:main} can be upgraded to an equality (the strong version) if we begin with the equality of the signed main conjecture. Comparing with the cyclotomic version \cite{kurihara-invent,kim-kurihara}, the correct formulation of the strong version would require to assume that $p$ does not divide any Tamagawa factor and to
put every $L_p(E/K_m)$ with $0 \leq m \leq n$ into the analytic side.
Also, since the argument for the strong version in \cite{kim-kurihara} depends heavily on the nature of Kato's zeta elements, we are not sure whether the same idea would apply even after having the correct formulation of the strong version.

\subsection{Organization}
In \S\ref{sec:bertolini-darmon-elements}, we review the construction of Bertolini--Darmon's theta elements and signed $p$-adic $L$-functions.
In \S\ref{sec:pm-iwasawa-theory}, we recall the basic objects of signed Iwasawa theory and review the standard results on signed Iwasawa theory. Then we observe how the finite layer specialization of the one-sided divisibility of the signed main conjecture looks like.
In \S\ref{sec:putting}, we compare the signed and unsigned objects in finite layers explicitly and deduce the main result.
\subsection*{Acknowledgement}
It is a great pleasure to dedicate this paper to Massimo Bertolini on the occasion of his 60th birthday for his great work and wonderful inspiration.
In my graduate years, the celebrated article \cite{bertolini-darmon-imc-2005} by Bertolini and Darmon was one of the very first serious research articles I read.
In particular, my thesis \cite{kim-summary} can be explained as ``a Greenberg--Vatsal type result for the Bertolini--Darmon setting".
It is evident that the content of \cite{bertolini-darmon-imc-2005} became an important part of my research.

I deeply thank Matteo Longo,  Marco Seveso, Stefano Vigni, and Rodolfo Venerucci for inviting me for his birthday conference and for giving me a chance to contribute this volume.

Indeed, the topic of this paper was exactly my original thesis problem suggested by Rob Pollack, but I could not find how to attack it at that time. 
Later, I found a way to resolve it in the middle of writing the paper with Masato Kurihara \cite{kim-kurihara}. I deeply appreciate Rob Pollack and Masato Kurihara.
This is the sequel to \cite{kim-mazur-tate}, which studies the ordinary case.

\section{Bertolini--Darmon's theta elements and anticyclotomic $p$-adic $L$-functions} \label{sec:bertolini-darmon-elements}
We quickly review the construction of Gross points of conductor $p^n$ and signed anticyclotomic $p$-adic $L$-functions. 
See  \cite{chida-hsieh-main-conj, chida-hsieh-p-adic-L-functions, kim-overconvergent} for details.
\subsection{Gross points}
Let $K$ be the imaginary quadratic field of discriminant $-D_K <0$.
Define
$$\vartheta := 
\left \lbrace
    \begin{array}{ll}
     \dfrac{ D_K - \sqrt{-D_K} }{ 2 }  & \textrm{ if } 2 \nmid D_K \\ 
\dfrac{ D_K - 2\sqrt{-D_K} }{ 4 }  & \textrm{ if } 2 \mid D_K
    \end{array}
    \right.
$$
 so that
$\mathcal{O}_K = \mathbb{Z} + \mathbb{Z}\vartheta$.
Let $B_{N^-}$ be the definite quaternion algebra over $\mathbb{Q}$ of discriminant $N^-$.
Then there exists an embedding of $K$ into $B_{N^-}$ \cite{vigneras}.
More explicitly, we choose a $K$-basis $(1,J)$ of $B_{N^-}$ so that $B_{N^-} = K \oplus K \cdot J$ such that
$\beta := J^2 \in \mathbb{Q}^\times$ with $\beta <0$, $J \cdot t = \overline{t} \cdot J$ for all $t \in K$, $\beta \in \left( \mathbb{Z}^\times_q \right)^2$ for all $q \mid pN^+$, and
$\beta \in \mathbb{Z}^\times_q$ for all $q \mid D_K$.
Fix a square root $\sqrt{\beta} \in \overline{\mathbb{Q}}$ of $\beta$.
For a $\mathbb{Z}$-module $A$, write $\widehat{A} = A \otimes \widehat{\mathbb{Z}}$.
Fix an isomorphism
$$i := \prod i_q : \widehat{B}^{(N^-)}_{N^-} \simeq \mathrm{M}_2(\mathbb{A}^{(N^-\infty)})$$
as follows:
\begin{itemize}
\item For each finite place $q \mid N^+p$, the isomorphism
$i_q : B_{N^-,q} \simeq \mathrm{M}_2(\mathbb{Q}_q)$ is defined by
\[
\xymatrix{
{
i_q(\vartheta)  = \left( \begin{matrix}
\mathrm{trd}(\vartheta) & - \mathrm{nrd}(\vartheta) \\
1 & 0
\end{matrix} \right)  }, & 
{
i_q(J)  = \sqrt{\beta} \cdot \left( \begin{matrix}
-1 & \mathrm{trd}(\vartheta) \\
0 & 1
\end{matrix} \right) 
}
}
\]
where $\mathrm{trd}$ and $\mathrm{nrd}$ are the reduced trace and the reduced norm on $B$, respectively.
\item For each finite place $q \nmid pN^+$, the isomorphism
$i_q : B_{N^-,q} \simeq \mathrm{M}_2(\mathbb{Q}_q)$ is chosen so that
$i_q \left( \mathcal{O}_K \otimes \mathbb{Z}_q  \right) \subseteq \mathrm{M}_2(\mathbb{Z}_q) $.
\end{itemize}
Under the fixed isomorphism $i$, for any rational prime $q$, the local Gross point $\varsigma_q \in B^\times_{N^-,q}$ is defined as follows:
\begin{itemize}
\item $\varsigma_q := 1$
in $B^\times_{nN^-,q}$ for $q \nmid pN^+$.
\item $\varsigma_q := \frac{1}{\sqrt{D_K}}\cdot \left( \begin{matrix}
\vartheta & \overline{\vartheta} \\
1 & 1
\end{matrix} \right) \in \mathrm{GL}_2(K_\mathfrak{q}) = \mathrm{GL}_2(\mathbb{Q}_q) $
for $q \mid N^+$ with $q = \mathfrak{q} \overline{\mathfrak{q}}$ in $\mathcal{O}_K$.
\item 
$\varsigma^{(n)}_p = \left( \begin{matrix}
\vartheta & -1 \\
1 & 0
\end{matrix} \right)
\cdot \left( \begin{matrix}
p^n & 0 \\
0 & 1
\end{matrix} \right)
 \in \mathrm{GL}_2(K_\mathfrak{p}) = \mathrm{GL}_2( \mathbb{Q}_{p} )$ where $p = \mathfrak{p}\overline{\mathfrak{p}}$ splits in $K$.
\end{itemize}
By using the fixed embedding of $K$ into $B_{N^-}$, we define
 $x_n : \widehat{K}^\times \to \widehat{B}^\times_{N^-}$
by
$x_n(a) = a \cdot \varsigma^{(n)} := a \cdot \left( \varsigma^{(n)}_p \times \prod_{q \neq p} \varsigma_q \right)$.
The collection $\left\lbrace x_n(a) : a \in \widehat{K}^\times \right\rbrace$ of points is called the \textbf{Gross points of conductor $p^n$ on $\widehat{B}^\times_{N^-}$}.
The fixed embedding $K \hookrightarrow B_{N^-}$ also induces an optimal embedding of $\mathcal{O}_n = \mathbb{Z}+p^n \mathcal{O}_K$ into the Eichler order
$B_{N^-} \cap \varsigma^{(n)}\widehat{R}_{N^+}(\varsigma^{(n)})^{-1}$ where $R_{N^+}$ is the Eichler order of level $N^+$ under the fixed isomorphism $i$.

\subsection{Theta elements and $p$-adic $L$-functions}
\subsubsection{}
Let $\phi_f : B^\times_{N^-} \backslash \widehat{B}^{\times}_{N^-} / \widehat{R}^\times_{N^+} \to \mathbb{C}$ be the Jacquet--Langlands transfer of $f$.
Since $B^\times_{N^-} \backslash \widehat{B}^{\times}_{N^-} / \widehat{R}^\times_{N^+}$ is a finite set and $f$ is a Hecke eigenform, we are able to and do normalize
$$\phi_f : B^\times_{N^-} \backslash \widehat{B}^{\times}_{N^-} / \widehat{R}^\times_{N^+} \to \mathbb{Z}_p$$
such that the image of $\phi_f $ does not lie in $p\mathbb{Z}_p$. This normalization is closely related to the $N^-$-new congruence ideal \cite{pw-mu, kim-summary,kim-ota}.
Let
$$\widetilde{\theta}_n(E/K) = \sum_{[a] \in \mathcal{G}_n} \phi_f (x_n(a) )  \cdot [a] \in \mathbb{Z}_p[\mathcal{G}_n]$$
where $\mathcal{G}_n = K^\times \backslash \widehat{K}^\times / \widehat{\mathcal{O}}^\times_n $ and $[a]$ is the image of $a \in \widehat{K}^\times$ in $\mathcal{G}_n$. Then
\textbf{Bertolini--Darmon's theta element $\theta_n(E/K)$ of $E$ over $K_n$} is defined by the image of $\widetilde{\theta}_n(E/K)$ in $\Lambda_n$
\[
\xymatrix@R=0em{
\mathbb{Z}_p[\mathcal{G}_n] \ar[r] & \Lambda_n\\
\widetilde{\theta}_n(E/K) \ar@{|->}[r] & \theta_n(E/K) .
}
\]
It is known that $\theta_n(E/K)$ interpolates ``an half of" $L(E, \chi, 1)$ where $ \chi$ runs over characters on $\mathrm{Gal}(K_n/K)$ \cite{chida-hsieh-p-adic-L-functions, kim-overconvergent}.
\subsubsection{}
Let $\omega_n = \omega_n(X) = (1+X)^{p^n} -1$ and $\Phi_n (1+X) = \omega_n(X) / \omega_{n-1}(X)$ where $\Phi_n$ is the $p^n$-th cyclotomic polynomial.
Fix a generator $\gamma$ of $\mathrm{Gal}(K_\infty/K)$ and take a generator $\gamma_n$ of $\mathrm{Gal}(K_n/K)$ as the image of $\gamma$. Then we have isomorphisms
$\Lambda_n  \simeq \mathbb{Z}_p[X]/ \left( \omega_n(X) \right)$
and
$\Lambda  \simeq \mathbb{Z}_p\llbracket X \rrbracket$
by sending the generators to $1+X$. Via the latter isomorphism, we also regard $\omega_n \in \Lambda$.
Let $\omega^{\pm}_0(X)  := X$, $\widetilde{\omega}^{\pm}_0(X)  := 1$, and
\[
\xymatrix@R=0em{
{\displaystyle \omega^{+}_n = \omega^{+}_n(X) := X\cdot \prod_{2 \leq m \leq n, m: \textrm{ even}}\Phi_m(1+X) } , &
{\displaystyle \omega^{-}_n = \omega^{-}_n(X)  := X\cdot \prod_{1 \leq m \leq n, m: \textrm{ odd}}\Phi_m(1+X) } , \\
{\displaystyle \widetilde{\omega}^{+}_n = \widetilde{\omega}^{+}_n(X)  := \prod_{2 \leq m \leq n, m: \textrm{ even}}\Phi_m(1+X) } , &
{\displaystyle \widetilde{\omega}^{-}_n = \widetilde{\omega}^{-}_n(X)  := \prod_{1 \leq m \leq n, m: \textrm{ odd}}\Phi_m(1+X) } .
}
\]
Then we have $\omega_n(X) = \omega^{\pm}_n(X) \cdot \widetilde{\omega}^{\mp}_n(X)$, respectively. We also regard
$\omega^{\pm}_n$, $\widetilde{\omega}^{\pm}_n$ as elements in $\Lambda_n$ or $\Lambda$.

\begin{prop}
Let $\epsilon$ be the sign of $(-1)^n$. Then:
\begin{enumerate}
\item $\omega^{\epsilon}_{n} \theta_n(E/K) = 0$.
\item There exists a unique element $\theta^{\epsilon}_n(E/K)$ in $\Lambda/\omega^{\epsilon}_n\Lambda$ such that 
\begin{equation} \label{eqn:betolini-darmon-theta-signed}
\theta_n(E/K) = \widetilde{\omega}^{-\epsilon}_{n}  \cdot \theta^{\epsilon}_n(E/K).
\end{equation}
\end{enumerate}
\end{prop}
\begin{proof}
See \cite[Prop. 2.8]{darmon-iovita}.
\end{proof}
Thanks to \cite[Lem. 2.9]{darmon-iovita}, $\left\lbrace (-1)^{n/2} \cdot \theta^{+}_n(E/K) : n \textrm{ is even}\right\rbrace$ and
 $\left\lbrace (-1)^{(n+1)/2} \cdot \theta^{-}_n(E/K) : n \textrm{ is odd} \right\rbrace$
 form compatible sequences with respect to $\Lambda/\omega^{+}_n\Lambda$ with even $n$ and $\Lambda/\omega^{-}_n\Lambda$ with odd $n$, respectively.
The signed $p$-adic $L$-functions for $E$ and $K$ are defined by 
\begin{align} \label{eqn:signed-p-adic-L-functions}
\begin{split}
L^+_p(E/K_\infty) & := \varprojlim_{n, \textrm{ even}} \left( (-1)^{n/2} \cdot \theta^{+}_n(E/K) \cdot  (-1)^{n/2} \cdot  \iota( \theta^{+}_n(E/K) ) \right) \\
& = \varprojlim_{n, \textrm{ even}} \left( \theta^{+}_n(E/K) \cdot  \iota( \theta^{+}_n(E/K) ) \right) \in \Lambda , \\
L^-_p(E/K_\infty) & := \varprojlim_{n, \textrm{ odd}} \left( (-1)^{(n+1)/2} \cdot \theta^{-}_n(E/K) \cdot (-1)^{(n+1)/2} \cdot   \iota( \theta^{-}_n(E/K) ) \right) \\
 & = \varprojlim_{n, \textrm{ odd}} \left(  \theta^{-}_n(E/K) \cdot   \iota( \theta^{-}_n(E/K) ) \right) \in \Lambda .
\end{split}
\end{align} 
 
\section{Anticyclotomic Iwasawa theory for elliptic curves with supersingular reduction} \label{sec:pm-iwasawa-theory}
\subsection{Basic objects of $\pm$-Iwasawa theory} \label{subsec:basic_objects}
We quickly recall the basic objects of $\pm$-Iwasawa theory, which is initiated by S. Kobayashi \cite{kobayashi-thesis} and R. Pollack \cite{pollack-thesis}.
See also \cite{iovita-pollack}.
\subsubsection{Local conditions at $p$} \label{subsubsec:local_condition_at_p}
Let $E$ be an elliptic curve over $\mathbb{Q}$ with $a_p(E) = 0$.
Since $p = \mathfrak{p} \cdot \overline{\mathfrak{p}}$ in $K$ and the (ram) condition, we are able to write
\begin{equation} \label{eqn:local-decomposition}
K_{n,p} = K_{n, \mathfrak{p}} \oplus K_{n, \overline{\mathfrak{p}}} \simeq \mathbb{Q}_{n,p} \oplus \mathbb{Q}_{n,p} 
\end{equation}
where $K_{n,p} = K_n \otimes \mathbb{Q}_p$ and $K_{n, v}$ is the completion of $K_{n}$ at $v \in \lbrace \mathfrak{p}, \overline{\mathfrak{p}} \rbrace$.
Then we define 
\begin{align*}
E^+(K_{n,p}) & := \lbrace P \in E(K_{n,p}) : \mathrm{Tr}_{n/m+1} (P) \in E(K_{m,p}) \textrm{ for even } m \ (0 \leq m < n)\rbrace \\
E^-(K_{n,p}) & := \lbrace P \in E(K_{n,p}) : \mathrm{Tr}_{n/m+1} (P) \in E(K_{m,p}) \textrm{ for odd } m \ (0 \leq m < n)\rbrace
\end{align*}
where $\mathrm{Tr}_{n/m+1} : E(K_{n,p}) \to E(K_{m+1,p})$ is the trace map.
\subsubsection{The norm subgroups}
Let $\widehat{E}$ be the formal group associated to $E$ and $\mathfrak{m}_n$ be the maximal ideal of $\mathbb{Q}_{n,p}$.
We define
\begin{align*}
\widehat{E}^+(\mathfrak{m}_{n}) & := \lbrace P \in \widehat{E}(\mathfrak{m}_{n}) : \mathrm{Tr}_{n/m+1} (P) \in \widehat{E}(\mathfrak{m}_{m}) \textrm{ for even } m \ (0 \leq m < n)\rbrace \\
\widehat{E}^-(\mathfrak{m}_{n}) & := \lbrace P \in \widehat{E}(\mathfrak{m}_{n}) : \mathrm{Tr}_{n/m+1} (P) \in \widehat{E}(\mathfrak{m}_{m}) \textrm{ for odd } m \ (0 \leq m < n)\rbrace
\end{align*}
where $\mathrm{Tr}_{n/m+1} : \widehat{E}(\mathfrak{m}_{n}) \to \widehat{E}(\mathfrak{m}_{m+1})$ is the trace map.
\subsubsection{$\pm$-Selmer groups}
Following  \cite{bdkim-supersingular-selmer}, we define the \textbf{$\pm$-Selmer groups of $E[p^\infty]$ over $K_n$} by
\begin{align*}
 \mathrm{Sel}^{\pm}(K_n, E[p^\infty]) : = & \mathrm{ker} \left( 
\mathrm{Sel}(K_n, E[p^\infty]) \to 
\dfrac{\mathrm{H}^1(K_{n,p}, E[p^\infty])}{E^{\pm}(K_{n,p}) \otimes \mathbb{Q}_p/\mathbb{Z}_p}
\right) 
\end{align*}
where $\mathrm{Sel}(K_n, E[p^\infty])$ is the standard Selmer group of $E[p^\infty]$ over $K_n$.
The \textbf{$\pm$-Selmer groups of $E[p^\infty]$ over $K_\infty$} is defined by
$$\mathrm{Sel}^{\pm}(K_\infty, E[p^\infty]) : = \varinjlim_n \mathrm{Sel}^{\pm}(K_n, E[p^\infty]) ,$$
respectively.

\subsection{Signed Iwasawa theory}
We recall the Euler system divisibility of the signed main conjectures.
\begin{thm} \label{thm:pm-main-conj}
Under (odd),  (Im), (ram), and (CR), we have the following statements:
\begin{enumerate}
\item $\mathrm{Sel}^{\pm}(K_\infty, E[p^\infty])$ is $\Lambda$-cotorsion.
\item $\left( L^{\pm}_p(E/K_\infty) \right) \subseteq \mathrm{char}_\Lambda \left( \mathrm{Sel}^{\pm}(K_\infty, E[p^\infty])^\vee \right)$.
\end{enumerate}
\end{thm}
\begin{proof}
See \cite{vatsal-uniform,darmon-iovita, pw-mu, kim-pollack-weston, burungale-buyukboduk-lei}.
All the conditions in Assumption \ref{assu:ram-CR} are needed to have this divisibility.
\end{proof}
The signed main conjecture claims that
\begin{equation} \label{eqn:main-conj}
\left( L^{\pm}_p(E/K_\infty) \right) = \mathrm{char}_\Lambda \left( \mathrm{Sel}^{\pm}(K_\infty, E[p^\infty])^\vee \right)
\end{equation}
We recall B.D.Kim's result on the non-existence of proper $\Lambda$-submodules of finite index.
See \cite{shii-non-trivial-Lambda-modules} for the $p$-inert case.
\begin{thm} \label{thm:no-finite}
If  $\mathrm{Sel}^{\pm}(K_\infty, E[p^\infty])$ is $\Lambda$-cotorsion, then
$\mathrm{Sel}^{\pm}(K_\infty, E[p^\infty])$ has no proper $\Lambda$-submodule of finite index, respectively.
\end{thm}
\begin{proof}
See \cite[Thm. 1.1]{bdkim-supersingular-selmer}.
\end{proof}
\begin{cor} \label{cor:no-finite}
Under (odd), (Im), (ram), and (CR),
$\mathrm{Sel}^{\pm}(K_\infty, E[p^\infty])$ has no proper $\Lambda$-submodule of finite index; thus, 
$$\mathrm{char}_{\Lambda} \mathrm{Sel}^{\pm}(K_\infty, E[p^\infty]) = \mathrm{Fitt}_{\Lambda} \mathrm{Sel}^{\pm}(K_\infty, E[p^\infty]) .$$
\end{cor}
\begin{proof}
It follows from Theorems \ref{thm:pm-main-conj} and \ref{thm:no-finite}.
\end{proof}
We recall the signed version of the control theorem.
\begin{thm} \label{thm:pm-control}
The restriction map
$$\mathrm{Sel}^{\pm}(K_n, E[p^\infty])[\omega^{\pm}_n] \to \mathrm{Sel}^{\pm}(K_\infty, E[p^\infty])[\omega^{\pm}_n] $$
is injective with the finite cokernel whose size is bounded independently of $n$. If we further assume that $p \nmid \mathrm{Tam}(E)$, then the restriction map
is an isomorphism.
\end{thm}
\begin{proof}
This is \cite[Thm. 6.8]{iovita-pollack}.
The $a_p(E)=0$ condition ensures that $E(K)[p]$ is trivial, so the injectivity of the restriction map follows as explained in \cite[Lem. 9.1]{kobayashi-thesis}.
Since the size of the cokernel depends only on prime-to-$p$ local conditions, the situation coincides with the ordinary case. By \cite[Lem. 3.3]{greenberg-lnm}, the cokernel is finite and is bounded independently of $n$.
When $p \nmid \mathrm{Tam}(E)$, the cokernel vanishes, so the restriction map becomes an isomorphism.
See also \cite[Thm. 9.3]{kobayashi-thesis} and \cite[Prop. 3.8]{greenberg-lnm} for further details.
\end{proof}

\subsection{The consequence}
\begin{cor} \label{cor:the-key-lemma}
Under  (odd), (Im), (ram), and (CR), we have
$$ \left(  \widetilde{\omega}^{\mp}_n  \cdot L^{\pm}_p(E/K_\infty) \Mod{\omega_n} \right)
\subseteq
\left( \widetilde{\omega}^{\mp}_n \right) \cdot \mathrm{Fitt}_{\Lambda_n} \left(  \mathrm{Sel}^{\pm}(K_n, E[p^\infty])^\vee \right) $$
 in $\Lambda_n$, respectively.
\end{cor}
\begin{proof}
By applying Theorem \ref{thm:pm-main-conj} and Corollary \ref{cor:no-finite}, we have
$$ \left( L^{\pm}_p(E/K_\infty) \right)  \subseteq \mathrm{Fitt}_\Lambda \left( \mathrm{Sel}^{\pm}(K_\infty, E[p^\infty])^\vee \right) .$$
Taking the quotient by $\omega^{\pm}_n$, we obtain
$$
\left( L^{\pm}_p(E/K_\infty) \Mod{\omega^{\pm}_n} \right) \subseteq \mathrm{Fitt}_{\Lambda_n / \omega^{\pm}_n} \left( \left( \mathrm{Sel}^{\pm}(K_\infty, E[p^\infty])[\omega^{\pm}_n] \right)^\vee \right)
$$
in $\Lambda_n/\omega^{\pm}_n$, respectively.
By the signed control theorem (Theorem \ref{thm:pm-control}), we obtain
$$\left( L^{\pm}_p(E/K_\infty) \Mod{\omega^{\pm}_n} \right) \subseteq \mathrm{Fitt}_{\Lambda_n / \omega^{\pm}_n} \left( \left( \mathrm{Sel}^{\pm}(K_n, E[p^\infty])[\omega^{\pm}_n] \right)^\vee \right)$$
in $\Lambda_n/\omega^{\pm}_n$, respectively.
Since Fitting ideals are compatible with base change, we have the equality
\begin{align*}
\mathrm{Fitt}_{\Lambda_n / \omega^{\pm}_n} \left( \left( \mathrm{Sel}^{\pm}(K_n, E[p^\infty])[\omega^{\pm}_n] \right)^\vee \right) & = \dfrac{ \mathrm{Fitt}_{\Lambda_n} \left(  \mathrm{Sel}^{\pm}(K_n, E[p^\infty])^\vee \right) + (\omega^{\pm}_n) }{ (\omega^{\pm}_n) }
\end{align*}
in $\Lambda_n / \omega^{\pm}_n$.
Thus, we have inclusions
$$\left( L^{\pm}_p(E/K_\infty) \Mod{\omega_n} \right)  + (\omega^{\pm}_n) \subseteq \mathrm{Fitt}_{\Lambda_n} \left(  \mathrm{Sel}^{\pm}(K_n, E[p^\infty])^\vee \right) + (\omega^{\pm}_n)$$
in $\Lambda_n$, respectively.
Multiplying $\widetilde{\omega}^{\mp}_n$, the conclusion immediately follows.
\end{proof}
\begin{rem}
If we further assume $p \nmid \mathrm{Tam}(E)$ and the $\pm$-main conjectures (\ref{eqn:main-conj}), then the inclusion in Corollary \ref{cor:the-key-lemma} becomes an equality.
\end{rem}

\section{The proof of Theorem \ref{thm:main}} \label{sec:putting}
By using the global Poitou--Tate duality, consider the exact sequence of $\Lambda_n$-modules (cf. \cite[(4.2)]{kitajima-otsuki})
\[
\xymatrix{
\left( \dfrac{E(K_{n,p}) \otimes \mathbb{Q}_p/\mathbb{Z}_p}{E^{\pm}(K_{n,p})  \otimes \mathbb{Q}_p/\mathbb{Z}_p} \right)^\vee \ar[r]^-{\iota^{\pm}} & \mathrm{Sel}(K_n, E[p^\infty])^\vee  \ar[r] & \mathrm{Sel}^\pm(K_n, E[p^\infty])^\vee  \ar[r] & 0 .
}
\]
Then we have
$$\mathrm{Fitt}_{\Lambda_n} \left( \left( \dfrac{E(K_{n,p}) \otimes \mathbb{Q}_p/\mathbb{Z}_p}{E^{\pm}(K_{n,p})  \otimes \mathbb{Q}_p/\mathbb{Z}_p} \right)^\vee / \mathrm{ker} (\iota^{\pm}) \right) \cdot \mathrm{Fitt}_{\Lambda_n} \left( \mathrm{Sel}^\pm(K_n, E[p^\infty])^\vee \right) 
 \subseteq
\mathrm{Fitt}_{\Lambda_n} \left( \mathrm{Sel}(K_n, E[p^\infty])^\vee \right) $$
thanks to the behavior of Fitting ideals in the short exact sequence.
Due to the behavior of Fitting ideals under the quotient map,  we also have
$$
\mathrm{Fitt}_{\Lambda_n} \left( \left( \dfrac{E(K_{n,p}) \otimes \mathbb{Q}_p/\mathbb{Z}_p}{E^{\pm}(K_{n,p})  \otimes \mathbb{Q}_p/\mathbb{Z}_p} \right)^\vee \right) \subseteq
\mathrm{Fitt}_{\Lambda_n} \left( \left( \dfrac{E(K_{n,p}) \otimes \mathbb{Q}_p/\mathbb{Z}_p}{E^{\pm}(K_{n,p})  \otimes \mathbb{Q}_p/\mathbb{Z}_p} \right)^\vee / \mathrm{ker} (\iota^{\pm}) \right) .
$$
Applying \cite[Lem. 3.14]{kitajima-otsuki} with (\ref{eqn:local-decomposition}), we observe that
\begin{align*}
\left( \dfrac{ E(K_{n,p}) \otimes \mathbb{Q}_p/\mathbb{Z}_p }{E^{\pm}(K_{n,p}) \otimes \mathbb{Q}_p/\mathbb{Z}_p} \right)^\vee & \simeq  \left( \dfrac{ \widehat{E}(\mathfrak{m}_{n}) \otimes \mathbb{Q}_p/\mathbb{Z}_p }{\widehat{E}^{\pm}(\mathfrak{m}_{n}) \otimes \mathbb{Q}_p/\mathbb{Z}_p} \right)^\vee \oplus \left( \dfrac{ \widehat{E}(\mathfrak{m}_{n}) \otimes \mathbb{Q}_p/\mathbb{Z}_p }{\widehat{E}^{\pm}(\mathfrak{m}_{n}) \otimes \mathbb{Q}_p/\mathbb{Z}_p} \right)^\vee \\
& \simeq \left( \dfrac{ \widehat{E}(\mathfrak{m}_{n})  }{\widehat{E}^{\pm}(\mathfrak{m}_{n}) } \otimes \mathbb{Q}_p/\mathbb{Z}_p \right)^\vee \oplus \left( \dfrac{ \widehat{E}(\mathfrak{m}_{n})  }{\widehat{E}^{\pm}(\mathfrak{m}_{n}) } \otimes \mathbb{Q}_p/\mathbb{Z}_p \right)^\vee .
\end{align*}
Due to \cite[Prop. 4.11]{iovita-pollack}, we have the following exact sequence
\[
\xymatrix{
0 \ar[r] & \widehat{E}(p\mathbb{Z}_p) \ar[r]^-{f} \ar[d]_-{\simeq} & \widehat{E}^+(\mathfrak{m}_n) \oplus \widehat{E}^-(\mathfrak{m}_n) \ar[r]^-{g} \ar[d]_-{\simeq} & \widehat{E}(\mathfrak{m}_n) \ar[r] \ar[d]_-{\simeq}^-{\textrm{\cite[Prop. 5.8]{iovita-pollack}}} & 0  \\
0 \ar[r] &  \widetilde{\omega}^{+}_n \widetilde{\omega}^{-}_n  \Lambda_n \ar[r] & \widetilde{\omega}^-_n\Lambda_n \oplus \widetilde{\omega}^+_n\Lambda_n \ar[r] & \left( \widetilde{\omega}^{+}_n, \widetilde{\omega}^{-}_n \right) \Lambda_n \ar[r] & 0
}
\]
where $f$ is the diagonal embedding and $g: (a,b) \mapsto a-b$.
Note that $\widetilde{\omega}^{+}_n \widetilde{\omega}^{-}_n  \Lambda_n \simeq \Lambda_n / X\Lambda_n  \simeq \mathbb{Z}_p  $.
Since
$\widehat{E}(\mathfrak{m}_n) / \widehat{E}^{\pm}(\mathfrak{m}_n) \simeq  \left( \widetilde{\omega}^{+}_n, \widetilde{\omega}^{-}_n \right) \Lambda_n / \widetilde{\omega}^{\mp}_n\Lambda_n$,
we have
$$\mathrm{Fitt}_{\Lambda_n} \left( \left( \frac{ \left( \widetilde{\omega}^{+}_n, \widetilde{\omega}^{-}_n \right) \Lambda_n }{ \widetilde{\omega}^{\mp}_n\Lambda_n } \otimes \mathbb{Q}_p/\mathbb{Z}_p \right)^\vee  \right)
=
\mathrm{Fitt}_{\Lambda_n} \left( \left(  \dfrac{ \widehat{E}(\mathfrak{m}_n) }{ \widehat{E}^{\pm}(\mathfrak{m}_n) }    \otimes \mathbb{Q}_p/\mathbb{Z}_p \right)^\vee  \right)  .$$
\begin{prop}
$$\mathrm{Fitt}_{\Lambda_n} \left( \left( \frac{ \left( \widetilde{\omega}^{+}_n, \widetilde{\omega}^{-}_n \right) \Lambda_n }{ \widetilde{\omega}^{\mp}_n\Lambda_n } \otimes \mathbb{Q}_p/\mathbb{Z}_p \right)^\vee  \right)= \widetilde{\omega}^{\mp}_n\Lambda_n,$$ respectively.
\end{prop}
\begin{proof}
See \cite[Prop. 4.1]{kim-kurihara}.
\end{proof}
To sum up, we have
\begin{align*}
& \left( \widetilde{\omega}^{\mp}_n \right)^2 \cdot \mathrm{Fitt}_{\Lambda_n} \left( \mathrm{Sel}^\pm(K_n, E[p^\infty])^\vee \right) \\
& =
\mathrm{Fitt}_{\Lambda_n} \left( \left( \frac{ \left( \widetilde{\omega}^{+}_n, \widetilde{\omega}^{-}_n \right) \Lambda_n }{ \widetilde{\omega}^{\mp}_n\Lambda_n } \otimes \mathbb{Q}_p/\mathbb{Z}_p \right)^\vee  \right)^{2} \cdot \mathrm{Fitt}_{\Lambda_n} \left( \mathrm{Sel}^\pm(K_n, E[p^\infty])^\vee \right) \\
& =
\mathrm{Fitt}_{\Lambda_n} \left( \left( \widehat{E}(\mathfrak{m}_n) / \widehat{E}^{\pm}(\mathfrak{m}_n) \otimes \mathbb{Q}_p/\mathbb{Z}_p \right)^\vee  \right)^{2} \cdot \mathrm{Fitt}_{\Lambda_n} \left( \mathrm{Sel}^\pm(K_n, E[p^\infty])^\vee \right) \\
& \subseteq
\mathrm{Fitt}_{\Lambda_n} \left( \left( \dfrac{E(K_{n,p}) \otimes \mathbb{Q}_p/\mathbb{Z}_p}{E^{\pm}(K_{n,p})  \otimes \mathbb{Q}_p/\mathbb{Z}_p} \right)^\vee / \mathrm{ker} (\iota^{\pm}) \right) \cdot \mathrm{Fitt}_{\Lambda_n} \left( \mathrm{Sel}^\pm(K_n, E[p^\infty])^\vee \right) \\
& \subseteq
\mathrm{Fitt}_{\Lambda_n} \left( \mathrm{Sel}(K_n, E[p^\infty])^\vee \right) .
\end{align*}
By Corollary \ref{cor:the-key-lemma} with the multiplication by $\widetilde{\omega}^{\mp}_n$, we have
$$\left( \left( \widetilde{\omega}^{\mp}_n \right)^2 \cdot L^{\pm}_p(E/K_\infty) \Mod{\omega_n} \right) \subseteq \mathrm{Fitt}_{\Lambda_n} \left( \mathrm{Sel}(K_n, E[p^\infty])^\vee \right) .$$
By using  (\ref{eqn:betolini-darmon-theta-signed}), (\ref{eqn:signed-p-adic-L-functions}), and the functional equation for Bertolini--Darmon's theta elements (see \cite[Prop. 2.13]{bertolini-darmon-mumford-tate-1996} and \cite[Lem. 1.5]{bertolini-darmon-imc-2005}), we have
\begin{align*}
\left( \left( \widetilde{\omega}^{\mp}_n \right)^2 \cdot L^{\pm}_p(E/K_\infty)  \Mod{\omega_n} \right) & = \left( \left( \widetilde{\omega}^{\mp}_n \right)^2 \cdot \theta^{\pm}(E/K_n) \cdot \iota(  \theta^{\pm}(E/K_n) )  \right) \\
& = \left( \left( \widetilde{\omega}^{\mp}_n \right)^2 \cdot \theta^{\pm}(E/K_n) \cdot \theta^{\pm}(E/K_n)  \right) \\
& = \left(  \theta(E/K_n) \cdot \theta(E/K_n)   \right) \\
& = \left(  \theta(E/K_n) \cdot \iota( \theta(E/K_n) )  \right) \\
& = \left(  L_p(E/K_n)  \right).
\end{align*}
Thus, Theorem \ref{thm:main} follows.

\bibliographystyle{amsalpha}
\bibliography{library}

\end{document}